\documentclass[12pt]{amsart}

\usepackage[margin=2.4cm]{geometry}
\usepackage{a4wide}

\usepackage{amsfonts,amsmath,amssymb,amsthm}
\usepackage{enumerate}

\newtheorem{lemma}{Lemma}[section]
\newtheorem{theorem}[lemma]{Theorem}
\newtheorem{corollary}[lemma]{Corollary}
\theoremstyle{remark}
\newtheorem{remark}[lemma]{Remark}
\newtheorem{problem}[lemma]{Problem}
\numberwithin{equation}{section}

\newcommand{\C}{\mathbb C}
\newcommand{\N}{\mathbb N}
\newcommand{\Root}{\mathcal R}
\newcommand{\R}{\mathbb R}
\newcommand{\rmi}{\mathrm{i}}
\newcommand*\xbar[1]{\hbox{\vbox{\hrule height 0.5pt \kern0.5ex\hbox{\kern-0.1em\ensuremath{#1}\kern-0.1em}}}}

\begin{document}

\author{Szymon Draga, Janusz Morawiec}
\title[Reducing the polynomial-like iterative equations order]{Reducing the polynomial-like iterative equations order and a generalized Zolt\'an Boros' problem}
\email[S. Draga]{szymon.draga@gmail.com}
\email[J. Morawiec]{janusz.morawiec@us.edu.pl}
\date{}
\address{Institute of Mathematics, University of Silesia, Bankowa 14, 40-007 Katowice, Poland}

\begin{abstract}
We present a technique for reducing the order of polynomial-like iterative equations; in particular, we answer a question asked by Wenmeng Zhang and Weinian Zhang. Our method involves the asymptotic behaviour of the sequence of consecutive iterates of the unknown function at a given point. As an application we solve a generalized problem of Zolt\'an Boros posed during the 50th ISFE (2012).
\end{abstract}

\keywords{polynomial-like equation, iterate, chracteristic equation, continuous solution, recurrence relation}
\subjclass[2010]{39B12, 39A06}

\maketitle

\section{Introduction}
Suppose $I\subset\R$ to be a non-degenerated interval and let $g\colon I\to I$ be a function. Assuming $g$ to be continuous we are interested in lowering the order of the equation
\begin{equation}\label{main-equation}
a_Ng^N(x)+\ldots+a_1g(x)+a_0x=0,
\end{equation}
where $N$ is a positive integer, the coefficients $a_0,a_1,\ldots,a_N$ are real and $a_0\neq0$; here and throughout this paper $g^n$ stands for the $n$-th iterate of $g$. Note that $a_0\neq0$ implies that $a_n\neq0$ for some $n=1,\ldots,N$. From now on we assume that $a_N\neq0$. It turns out that continuous solutions to \eqref{main-equation} deeply depend on the roots of its characteristic equation
\begin{equation}\label{characteristic-equation}
a_Nr^N+\ldots+a_1r+a_0=0,
\end{equation}
which usually is obtained by assuming that $g$ has the form $g(x)=rx$. Up to now the case $N=2$ is the only non-trivial one which has been completely solved (see \cite{nabeya}). In fact, the problem still remains open even for $N=3$ (see \cite{matkowski}). These difficulties follow from the non-linearity of the operator $g\mapsto g^n$. Nonetheless, a lot of investigation was done in this matter; see a survey on functional equations with superpositions of the unknown function \cite[Section 3]{baron-jarczyk} and a survey on iterative equations of polynomial type \cite{zhang-yang-zhang}.

One of methods for finding solutions to equation \eqref{main-equation}, and also to its non-homogenous counterpart, where zero on the right-hand side is replaced by an arbitrary continuous function, is based on lowering its order. First such results on the whole real line were obtained in \cite{matkowski-zhang1} in the case where all roots of the characteristic equation are real and satisfy some special conditions. Further research in this direction was done in \cite{yang-zhang,zhang-gong,zhang-zhang}, but still most cases remain unsolved. For some exploration of equation \eqref{main-equation} on intervals see \cite{li-zhang,mukherjea-ratti,ratti-lin,zhang-baker,zhang-xu-zhang} and on half-lines see \cite{jarczyk}.

Let us note that \eqref{characteristic-equation} also can be considered as the characteristic equation of the recurrence relation
\begin{equation}\label{recurrence-relation}
a_Nx_{m+N}+\ldots+a_1x_{m+1}+a_0x_m=0
\end{equation}
which might be obtained by choosing $x_0\in I$ arbitrarily and putting $x_m=g(x_{m-1})$ for all $m\in\N$. Such an approach we will examine in the present paper.

It can be easily observed that if a polynomial $b_Mr^M+\ldots+b_1r+b_0$ divides a polynomial $a_Nr^N+\ldots+a_1r+a_0$ and a function $g$ satisfies the equation
\begin{equation}\label{reduced-equation}
b_Mg^M(x)+\ldots+b_1g(x)+b_0x=0,
\end{equation}
then it satisfies also \eqref{main-equation}; for a simple proof see \cite{matkowski-zhang2}. The main objective of this paper is to give particular conditions under which a converse holds, i.e., we want to find some conditions guaranteeing that if a continuous function $g$ satisfies \eqref{main-equation}, then there is a divisor of $a_Nr^N+\ldots+a_1r+a_0$ such that $g$ is a solution to the corresponding iterative equation of lower order. Of course, it is not possible in each case. For instance, we may consider the equation $g^2(x)-2g(x)+x=0$ whose continuous solution (on the whole real line) is $g(x)=x+c$, where $c$ is a constant.

The paper is organized as follows. Section \ref{preliminaries} contains basic properties of solutions to considered equations and preliminary information on linear homogenous recurrence relations. In Section \ref{nonreal-elimination} we prove theorems on eliminating non-real roots from the characteristic equation which show that, in new crucial cases, equation \eqref{main-equation} is equivalent to an equation of lower order. In particular, we generalize results from \cite{zhang-zhang} and give an answer to the question posed in the last section of that paper. In Section \ref{opposite-elimination} we obtain similar results as in the previous section, but we eliminate real roots of opposite sign. These theorems allow us to solve, in Section \ref{boros-section}, a generalized problem of Zolt\'an Boros. Namely, for a given integer $n\ge3$ we determine all continuous self-mapping functions $f$ acting on an interval and satisfying
\begin{equation}\label{zoltan-equation}
f^n(x)=\frac{[f(x)]^n}{x^{n-1}}.
\end{equation}
For $n=2$ the equation above can be solved using the mentioned result of Nabeya or \cite{morawiec}. In the last section we discuss further problems.

\section{Preliminaries}\label{preliminaries}
As we mentioned before we assume that $a_0\neq0$.

\begin{lemma}\label{lemma-injective}
If a function $g$ satisfies \eqref{main-equation}, then it is injective.
\end{lemma}

\begin{proof}
Choose $x,y\in I$ and suppose that $g(x)=g(y)$. Then
\[ x=-\frac{1}{a_0}\left(a_Ng^N(x)+\ldots+a_1g(x)\right)=-\frac{1}{a_0}\left(a_Ng^N(y)+\ldots+a_1g(y)\right)=y.\qedhere \]
\end{proof}

The foregoing lemma directly implies that every continuous solution to \eqref{main-equation} is strictly monotone. It means that the sequence $(x_m)_{m\in\N_0}$ given by $x_0\in I$ and $x_m=g(x_{m-1})$ for all $m\in\N$ is either monotone (in the case of increasing $g$) or anti-monotone (in the case of decreasing $g$). By \textit{anti-monotone} we mean that the expression $(-1)^m(x_m-x_{m-1})$ does not change its sign when $m$ runs through $\N_0$.

In the case where $g$ is bijective $h=g^{-1}$ is also a self-mapping function acting on $I$. Putting $g^{-N}(x)$ in place of $x$ in \eqref{main-equation} we obtain the equation
\begin{equation}\label{dual-equation}
a_0h^N(x)+\ldots+a_{N-1}h(x)+a_Nx=0
\end{equation}
which is called the dual equation. It is worth mentioning that if $r_1,\ldots,r_N$ are roots of \eqref{characteristic-equation}, then the roots of the characteristic equation of \eqref{dual-equation} are $r_1^{-1},\ldots,r_N^{-1}$. Whence if a bijective function $g$ satisfies \eqref{main-equation} and its inverse satisfies an iterative equation of a lower order, then the order of \eqref{main-equation} can be reduced. Moreover, a root from the characteristic equation can be eliminated if and only if its inverse can be eliminated from the characteristic equation of the dual equation.

\begin{lemma}\label{lemma-recurrence}
A function $g$ satisfies equation \eqref{main-equation} if and only if for every $x_0\in I$ the sequence $(x_m)_{m\in\N_0}$ given by $x_m=g(x_{m-1})$ for all $m\in\N$ satisfies recurrence relation \eqref{recurrence-relation}.
\end{lemma}

For the theory of linear recurrence relations see, for instance, \cite[\S 3.2]{jerri}. We will recall only the most significant theorem in this matter. In order to do this and simplify the writing we introduce the following notation. For a given polynomial $c_Kr^K+\ldots+c_1r+c_0$ we denote by $\Root(c_K,\ldots,c_0)$ the set $\{(r_1,k_1),\ldots,(r_p,k_p)\}$ of all pairs of pairwise distinct roots $r_1,\ldots,r_p$ and their multiplicities $k_1,\ldots,k_p$, respectively. Here and throughout the paper by a polynomial we mean a polynomial with real coefficients. Note that in the introduced notation $k_1+\ldots+k_p$ equals the degree of $c_Kr^K+\ldots+c_1r+c_0$. Moreover, $(\mu,k),(\xbar{\mu},k)\in\Root(c_K,\ldots,c_0)$ forces $\mu$ to be non-real.

\begin{theorem}\label{theorem-recurrence}
Assume
\[ \Root(a_N,\ldots,a_0)=\{(\lambda_1,l_1)\ldots,(\lambda_p,l_p),(\mu_1,k_1),(\xbar{\mu_1},k_1),\ldots,(\mu_q,k_q),(\xbar{\mu_q},k_q)\}. \]
Then a real-valued sequence $(x_m)_{m\in\N_0}$ is a solution to \eqref{recurrence-relation} if and only if it is given by
\[ x_m=\sum_{k=1}^pA_k(m)\lambda_k^m+\sum_{j=1}^q(B_j(m)\cos m\phi_j+C_j(m)\sin m\phi_j)|\mu_j|^m\quad\mbox{ for all }m\in\N_0, \]
where $A_k$ is a polynomial whose degree equals at most $l_k-1$ for $k=1,\ldots,p$ and $B_j,C_j$ are polynomials whose degrees equal at most $k_j-1$, with $\phi_j$ being an argument of $\mu_j$, for $j=1,\ldots,q$.
\end{theorem}

\section{Eliminating non-real roots}\label{nonreal-elimination}
We will make use of the following lemma whose proof can be found in \cite{tabor-tabor}.

\begin{lemma}\label{lemma-tabor}
Let the sequence $(x_m)_{m\in\N_0}$ be given by $x_m=\sum_{n=1}^N (a_n\cos m\phi_n+b_n\sin m\phi_n)$ for all $m\in\N_0$ with $a_n, b_n\in\R$ and $\phi_n\in(0,2\pi)$ for $n=1,\ldots,N$. If $\liminf_{m\to\infty}x_m\ge0$, then $x_m=0$ for all $m\in\N_0$.
\end{lemma}

\begin{lemma}\label{lemma-complex}
Suppose $\lambda_1,\ldots,\lambda_N$ to be complex numbers such that $\lambda_n=r(\cos\phi_n+\rmi\sin\phi_n)$ with $r>0$ and $\phi_n\in(0,\pi)$ for $n=1,\ldots,N$. Morever, let $F\colon\N_0\to\R$ be a function such that $\lim_{m\to\infty}F(m)/r^m=0$ and let $A_n,B_n$ be polynomials for $n=1,\ldots,N$. If the sequence $(x_m)_{m\in\N_0}$ given by
\[ x_m=F(m)+\sum_{n=1}^N r^m(A_n(m)\cos m\phi_n+B_n(m)\sin m\phi_n)\quad\mbox{ for all }m\in\N_0 \]
is either monotone or anti-monotone, then $x_m=F(m)$ for all $m\in\N_0$.
\end{lemma}

\begin{proof}
Assume that $k$ is the greatest number from degrees of the polynomials $A_1,\ldots,A_N$ and $B_1,\ldots,B_N$. We may write $A_n(m)=\sum_{l=0}^k a^{(n)}_lm^l$ and $B_n(m)=\sum_{l=0}^k b^{(n)}_lm^l$ for $n=1,\ldots,N$, where some from the coefficients $a^{(n)}_k$'s and $b^{(n)}_k$'s are possibly zeros.

We first consider the case where $(x_m)_{m\in\N_0}$ is monotone. By considering the sequence $(-x_m)_{m\in\N_0}$ if necessary, we may asssume that $x_m\ge0$ for all but finitely many $m\in\N_0$. It means that the inequality
\[ \frac{F(m)}{r^m\cdot m^k}+\sum_{n=1}^N \left(\frac{A_n(m)}{m^k}\cos m\phi_n +\frac{B_n(m)}{m^k}\sin m\phi_n\right) \ge0 \]
holds for all but finitely many $m\in\N_0$. Consequently, we obtain
\[ \liminf_{m\to\infty}\sum_{n=1}^N \left(\frac{A_n(m)}{m^k}\cos m\phi_n+\frac{B_n(m)}{m^k}\sin m\phi_n\right) \ge0 \]
and hence
\[ \liminf_{m\to\infty}\sum_{n=1}^N \left(a_k^{(n)}\cos m\phi_n+b_k^{(n)}\sin m\phi_n\right) \ge0. \]
In view of Lemma \ref{lemma-tabor}
\begin{equation}\label{complex-step}
\sum_{n=1}^N \left(a_k^{(n)}\cos m\phi_n+b_k^{(n)}\sin m\phi_n\right)=0\quad\mbox{ for all }m\in\N_0,
\end{equation}
which obviously implies
\[\sum_{n=1}^N \left(a_k^{(n)}m^k\cos m\phi_n+b_k^{(n)}m^k\sin m\phi_n\right)=0\quad\mbox{ for all }m\in\N_0.\]
It means that we can eliminate terms with $m^k$ from $\sum_{n=1}^N (A_n(m)\cos m\phi_n+B_n(m)\sin m\phi_n)$ and now we can assume that the greatest number from degrees of the polynomials $A_n$'s and $B_n$'s is equal to $l\le k-1$. The same procedure shows that the terms with $m^l$ also can be eliminated. Continuing in this fashion, we obtain
\[ \sum_{n=1}^N (A_n(m)\cos m\phi_n+B_n(m)\sin m\phi_n)=0\quad\mbox{ for all }m\in\N_0. \]

Secondly, we consider the case where $(x_m)_{m\in\N_0}$ is anti-monotone. By considering the sequence $(-x_m)_{m\in\N_0}$ if necessary, we may asssume that $(-1)^m(x_m-x_{m-1})\ge0$ for all $m\in\N$. It implies that the inequality
\begin{multline*}
\frac{F(2m)-F(2m-1)}{r^{2m}\cdot (2m)^k}+\sum_{n=1}^N \left(\frac{A_n(2m)}{(2m)^k}\cos 2m\phi_n+\frac{B_n(2m)}{(2m)^k}\sin 2m\phi_n\right)\\
-\frac1r\sum_{n=1}^N \left(\frac{A_n(2m-1)}{(2m)^k}\cos(2m-1)\phi_n+\frac{B_n(2m-1)}{(2m)^k}\sin(2m-1)\phi_n\right) \ge0
\end{multline*}
holds for all $m\in\N$ and, as previously,
\begin{equation}\label{complex-tabor}
\begin{aligned}
\liminf_{m\to\infty} \sum_{n=1}^N \Big( &a_k^{(n)}\cos2m\phi_n+b_k^{(n)}\sin2m\phi_n\\
&-\frac1r a_k^{(n)}\cos(2m-1)\phi_n-\frac1r b_k^{(n)}\sin(2m-1)\phi_)\Big)\ge0.
\end{aligned}
\end{equation}
Since we have $\cos(2m-1)\phi_n=\cos\phi_n\cos2m\phi_n+\sin\phi_n\sin2m\phi_n$ and $\sin(2m-1)\phi_n=\cos\phi_n\sin2m\phi_n-\sin\phi_n\cos2m\phi_n$, limit \eqref{complex-tabor} is of the form as in Lemma \ref{lemma-tabor}. Therefore,
\[ \sum_{n=1}^N \left(a_k^{(n)}\cos2m\phi_n+b_k^{(n)}\sin2m\phi_n-\frac1r a_k^{(n)}\cos(2m-1)\phi_n-\frac1r b_k^{(n)}\sin(2m-1)\phi_n\right)=0 \]
and, consequently, the equality
\[ \sum_{n=1}^N \left(a_k^{(n)}\cos2m\phi_n+b_k^{(n)}\sin2m\phi_n\right)=\frac1r\sum_{n=1}^N \left(a_k^{(n)}\cos(2m-1)\phi_n+b_k^{(n)}\sin(2m-1)\phi_n\right) \]
holds for all $m\in\N$. In the same manner we can show that
\[ \sum_{n=1}^N \left(a_k^{(n)}\cos2m\phi_n+b_k^{(n)}\sin2m\phi_n\right)=r\sum_{n=1}^N \left(a_k^{(n)}\cos(2m+1)\phi_n+b_k^{(n)}\sin(2m+1)\phi_n\right). \]
Therefore, the sequence
\[ \left(r^m\sum_{n=1}^N \big(a_k^{(n)}\cos m\phi_n+b_k^{(n)}\sin m\phi_n\big)\right)_{m\in\N_0} \]
is constant. In particular, $\sum_{n=1}^N \big(a_k^{(n)}\cos m\phi_n+b_k^{(n)}\sin m\phi_n\big)$ has a constant sign. Applying Lemma \ref{lemma-tabor} once again, we conclude that \eqref{complex-step} holds. Further reasoning is exactly the same as in the foregoing case.
\end{proof}

Now we are in a position to prove one of our main results.
\begin{theorem}\label{complex-thm1}
Assume
\[ \Root(a_N,\ldots,a_0)=\{(\lambda_1,l_1)\ldots,(\lambda_p,l_p),(\mu_1,k_1),(\xbar{\mu_1},k_1),\ldots,(\mu_q,k_q),(\xbar{\mu_q},k_q)\}. \]
If 
\begin{equation}\label{complex-inequality1}
|\lambda_1|\le\ldots\le|\lambda_p|<|\mu_1|\le\ldots\le|\mu_q|,
\end{equation}
then a continuous function $g\colon I\to I$ satisfies equation \eqref{main-equation} if and only if it satisfies \eqref{reduced-equation} with
\[ \Root(b_M,\ldots,b_0)=\{(\lambda_1,l_1)\ldots,(\lambda_p,l_p)\}. \]
\end{theorem}

\begin{proof}
Fix $x\in\R$ and define a sequence $(x_m)_{m\in\N_0}$ by putting
\begin{equation}\label{seq-def}
x_0=x\quad\mbox{ and }\quad x_{m}=g(x_{m-1})\quad\mbox{ for all }m\in\N.
\end{equation}
By Lemma \ref{lemma-recurrence} this sequence satisfies equation \eqref{recurrence-relation}. Therefore, by Theorem \ref{theorem-recurrence}, it is of the form
\[ x_m=F(m)+\sum_{n=1}^q |\mu_n|^m(A_n(m)\cos m\phi_n+B_n(m)\sin m\phi_n)\quad\mbox{ for all }m\in\N_0; \]
here $F$ stands for the part of the solution for which the roots $\lambda_1,\ldots,\lambda_p$ are responsible, $\phi_n$ is the principal argument of $\mu_n$ and $A_n,B_n$ are polynomials for $n=1,\ldots,q$. Applying Lemma \ref{lemma-complex} to $(x_m)_{m\in\N_0}$ one can eliminate the roots with modulus $|\mu_q|$. In each another step one can eliminate the non-real roots with the greatest modulus. This procedure yields $x_m=F(m)$ for all $m\in\N_0$. Hence, again by Theorem \ref{theorem-recurrence}, $(x_m)_{m\in\N_0}$ satisfies
\[ b_Mx_{m+M}+\ldots+b_1x_{m+1}+b_0x_m=0 \]
and, consequently, \eqref{reduced-equation} holds.
\end{proof}

\begin{corollary}\label{complex-cor1}
Assume
\[ \Root(a_N,\ldots,a_0)=\{(\lambda_1,l_1)\ldots,(\lambda_p,l_p),(\mu_1,k_1),(\xbar{\mu_1},k_1),\ldots,(\mu_q,k_q),(\xbar{\mu_q},k_q)\}. \]
If
\begin{equation}\label{complex-inequality2}
|\lambda_1|\ge\ldots\ge|\lambda_p|>|\mu_1|\ge\ldots\ge|\mu_q|,
\end{equation}
then a continuous surjection $g\colon I\to I$ satisfies equation \eqref{main-equation} if and only if it satisfies \eqref{reduced-equation} with \[ \Root(b_M,\ldots,b_0)=\{(\lambda_1,l_1)\ldots,(\lambda_p,l_p)\}. \]
\end{corollary}

\begin{proof}
By Lemma \ref{lemma-injective} case \eqref{complex-inequality2} may be reduced to case \eqref{complex-inequality1} by considering the dual equation. Thus, the assertion follows from Theorem \ref{complex-thm1}.
\end{proof}

\begin{remark}\label{remark-surjective}
It is well-known (see e.g. \cite{zhang-zhang}) that if $I=\R$, then each continuous solution to \eqref{main-equation} maps $\R$ onto $\R$. Hence the assumption of surjectivity in Corollary \ref{complex-cor1} is satisfied automatically in the case when $I=\R$.
\end{remark}

\begin{theorem}\label{complex-thm2}
Assume
\[ \Root(a_N,\ldots,a_0)=\{(1,1),(\lambda_1,l_1)\ldots,(\lambda_p,l_p),(\mu_1,k_1),(\xbar{\mu_1},k_1),\ldots,(\mu_q,k_q),(\xbar{\mu_q},k_q)\}. \]
If
\[ |\lambda_1|\le\ldots\le|\lambda_p|<|\mu_1|\le\ldots\le|\mu_q|, \]
then a continuous function $g\colon I\to I$ satisfies equation \eqref{main-equation} if and only if it satisfies \eqref{reduced-equation} with
\[ \Root(b_M,\ldots,b_0)=\{(1,1),(\lambda_1,l_1)\ldots,(\lambda_p,l_p)\}. \]
\end{theorem}

\begin{proof}
The proof is very similar to the proof of Theorem \ref{complex-thm1}. Here the sequence $(x_m)_{m\in\N_0}$ has the form
\[ x_m=C+F(m)+\sum_{n=1}^q |\mu_n|^m(A_n(m)\cos m\phi_n+B_n(m)\sin m\phi_n)\quad\mbox{ for all }m\in\N_0, \]
where $C$ is a constant. Since a constant has no effect on the monotonicity or the anti-monotonicity of a sequence, we can apply Lemma \ref{lemma-complex} and argue as before.
\end{proof}

\begin{corollary}\label{complex-cor2}
Assume
\[ \Root(a_N,\ldots,a_0)=\{(1,1),(\lambda_1,l_1)\ldots,(\lambda_p,l_p),(\mu_1,k_1),(\xbar{\mu_1},k_1),\ldots,(\mu_q,k_q),(\xbar{\mu_q},k_q)\}. \]
If
\[ |\lambda_1|\ge\ldots\ge|\lambda_p|>|\mu_1|\ge\ldots\ge|\mu_q|, \]
then a continuous surjection $g\colon I\to I$ satisfies equation \eqref{main-equation} if and only if it satisfies \eqref{reduced-equation} with
\[ \Root(b_M,\ldots,b_0)=\{(1,1),(\lambda_1,l_1)\ldots,(\lambda_p,l_p)\}. \]
\end{corollary}

As an immediate corollary from the theorems above we obtain a result which for $I=\R$ was also proven in \cite{yang-zhang}.

\begin{corollary}
If all the roots of characteristic equation \eqref{characteristic-equation} are non-real, then equation \eqref{main-equation} has no continuous solution $g\colon I\to I$.
\end{corollary}

\section{Eliminating roots of opposite sign}\label{opposite-elimination}
Below we are going to prove that the order of polynomial-like equation \eqref{main-equation} also can be lowered when the minimal and the maximal (with respect to the absolute value) root of its characteristic equation are real and of opposite sign. However, it is necessary to distinguish two cases---when the unknown function monotonically increases and when it monotonically decreases.

\begin{theorem}\label{theorem-opposite1a}
Assume
\[ \Root(a_N,\ldots,a_0)=\{(r_1,k_1),(r_2,k_2),(\lambda_1,l_1),\ldots,(\lambda_p,l_p)\} \]
If
\[ |r_1|<|\lambda_1|\le\ldots\le|\lambda_p|<|r_2| \]
and $r_1$, $r_2$ are real with $r_1r_2<0$, then a continuous increasing surjection $g\colon I\to I$ satisfies equation \eqref{main-equation} if and only if it satisfies \eqref{reduced-equation} with
\[ \Root(b_M,\ldots,b_0)=\{(r_j,k_j),(\lambda_1,l_1),\ldots,(\lambda_p,l_p)\}, \]
where $j$ is the index of the positive number from $r_1$ and $r_2$.
\end{theorem}

\begin{proof}
Fix $x\in\R$ and define a sequence $(x_m)_{m\in\N_0}$ by \eqref{seq-def}. By Lemma \ref{lemma-recurrence} this sequence satisfies equation \eqref{recurrence-relation}. By replacing equation \eqref{main-equation} with the dual equation if necessary, we may assume that $r_2<0$. Since $g$ is increasing, $(x_m)_{m\in\N_0}$ is monotone and, by Theorem \ref{theorem-recurrence}, of the form
\[ x_m=F(m)+A(m)r_2^m\quad\mbox{ for all }m\in\N_0, \]
where $F$ represents the part of the solution to \eqref{recurrence-relation} for which the roots $r_1,\lambda_1,\ldots,\lambda_p$ are responsible and $A$ is a polynomial of degree at most $k_2-1$. Since
\[ x_m=m^{k_2-1}\cdot|r_2|^m\left(\frac{F(m)}{m^{k_2-1}\cdot|r_2|^m}+\frac{A(m)}{m^{k_2-1}}(-1)^m\right)\quad\mbox{ for all }m\in\N_0 \]
and
\[ \lim_{m\to\infty}\frac{F(m)}{m^{k_2-1}\cdot|r_2|^m}=0,\quad\lim_{m\to\infty}\frac{A(m)}{m^{k_2-1}}=a_{k_2-1}, \]
where $a_{k_2-1}$ stands for the coefficient at $m^{k_2-1}$ in the polynomial $A$, the monotonicity of $(x_m)_{m\in\N_0}$ forces $a_{k_2-1}=0$. Continuing in this fashion, we eliminate all non-zero terms from $A$ and obtain $A\equiv0$. It means that $x_m=F(m)$ for all $m\in\N_0$ and, again by Theorem \ref{theorem-recurrence}, it safisfies the relation $b_M x_{m+M}+\ldots+b_1x_{m+1}+b_0x_m=0$, which ends the proof.
\end{proof}

\begin{theorem}\label{theorem-opposite1b}
Assume
\[ \Root(a_N,\ldots,a_0)=\{(r_1,k_1),(r_2,k_2),(\lambda_1,l_1),\ldots,(\lambda_p,l_p)\}. \]
If
\[ |r_1|<|\lambda_1|\le\ldots\le|\lambda_p|<|r_2| \]
and $r_1$, $r_2$ are real with $r_1r_2<0$, then a continuous decreasing surjection $g\colon I\to I$ satisfies equation \eqref{main-equation} if and only if
\begin{enumerate}
\item[\rm(i)] it satisfies \eqref{reduced-equation} with
\[ \Root(b_M,\ldots,b_0)=\{(r_j,k_j),(\lambda_1,l_1),\ldots,(\lambda_p,l_p)\}, \]
where $j$ is the index of the negative number from $r_1$ and $r_2$, in the case where $r_1\neq1\neq r_2$
\end{enumerate}
or
\begin{enumerate}
\item[\rm(ii)] it satisfies \eqref{reduced-equation} with
\[ \Root(b_M,\ldots,b_0)=\{(1,1),(r_j,k_j),(\lambda_1,l_1),\ldots,(\lambda_p,l_p)\}, \]
where $j$ is the index of the negative number from $r_1$ and $r_2$, in the case where the positive root from $r_1$ and $r_2$ equals~$1$.
\end{enumerate}
\end{theorem}

\begin{proof}
Fix $x\in\R$ and define a sequence $(x_m)_{m\in\N_0}$ by \eqref{seq-def}. By Lemma \ref{lemma-recurrence} this sequence satisfies equation \eqref{recurrence-relation}.

Firstly, we consider the case where $r_1\neq1\neq r_2$. By replacing equation \eqref{main-equation} with the dual equation if necessary, we may assume that $r_2>0$. Since $g$ is decreasing, $(x_m)_{m\in\N_0}$ is anti-monotone and of the form
\[ x_m=F(m)+A(m)r_2^m\quad\mbox{ for all }m\in\N_0, \]
where $F$ represents the part of the solution to \eqref{recurrence-relation} for which the roots $r_1,\lambda_1,\ldots,\lambda_p$ are responsible and $A$ is a polynomial of degree at most $k_2-1$. If it were $A\not\equiv0$, we would have
\[ x_{m+1}-x_m=A(m)\cdot r_2^m\left(\frac{F(m+1)-F(m)}{A(m)\cdot r_2^m}+\frac{A(m+1)}{A(m)}\,r_2-1\right)\quad\mbox{ for all }m\in\N_0 \]
and
\[ \lim_{m\to\infty}\left(\frac{F(m+1)-F(m)}{A(m)\cdot r_2^m}+\frac{A(m+1)}{A(m)}\,r_2-1\right)=r_2-1\neq0, \]
which is a contradiction to the anti-monotonicity of $(x_m)_{m\in\N_0}$. Consequently, $x_m=F(m)$ for all $m\in\N_0$ and, similarly as in the previous proof, \eqref{reduced-equation} is satisfied with $b_j$'s as in assertion (i).

Secondly, we consider the case where either $r_1=1$ or $r_2=1$. By replacing equation \eqref{main-equation} with the dual equation if necessary, we may assume that $r_2=1$. Again $(x_m)_{m\in\N_0}$ is anti-monotone and of the form
\[ x_m=F(m)+A(m)\quad\mbox{ for all }m\in\N_0, \]
where $F$ represents the part of the solution to \eqref{recurrence-relation} for which the roots $r_1,\lambda_1,\ldots,\lambda_p$ are responsible and $A$ is a polynomial of degree at most $k_2-1$. If the polynomial $A(m+1)-A(m)$ were not constant, then denoting by $k$ its degree, we would have
\[ x_{m+1}-x_m=m^k\left(\frac{F(m+1)-F(m)}{m^k}+\frac{A(m+1)-A(m)}{m^k}\right) \]
and
\[ \lim_{m\to\infty}\left(\frac{F(m+1)-F(m)}{m^k}+\frac{A(m+1)-A(m)}{m^k}\right)=a_k\neq0, \]
where $a_k$ stands for the coefficient at $m^k$ in $A(m+1)-A(m)$, which is a contradiction to the anti-monotonicity of $(x_m)_{m\in\N_0}$. Consequently, $A$ is constant and \eqref{reduced-equation} holds with $b_j$'s as in assertion (ii).
\end{proof}

Reasoning in the same manner as in the proof of Theorem \ref{complex-thm2} and making use of the proof above, one can show that a single root $1$ does not prevent the elimination of other roots. More precisely, we have the following result.

\begin{theorem}\label{theorem-oppositeii}
Assume
\[ \Root(a_N,\ldots,a_0)=\{(1,1),(r_1,k_1),(r_2,k_2),(\lambda_1,l_1),\ldots,(\lambda_p,l_p)\}. \]
If
\[ |r_1|<|\lambda_1|\le\ldots\le|\lambda_p|<|r_2| \]
and $r_1$, $r_2$ are real with $r_1r_2<0$, then a continuous decreasing surjection $g\colon I\to I$ satisfies equation \eqref{main-equation} if and only if it satisfies \eqref{reduced-equation} with
\[ \Root(b_M,\ldots,b_0)=\{(1,1),(r_j,k_j),(\lambda_1,l_1),\ldots,(\lambda_p,l_p)\}, \]
where $j$ is either the index of the negative number from $r_1$ and $r_2$, in the case of decreasing $g$, or the index of the positive number from $r_1$ and $r_2$, in the case of increasing $g$.
\end{theorem}

\begin{remark}\label{remark-opposite}
The assumption of surjectivity in Theorems \ref{theorem-opposite1a}, \ref{theorem-opposite1b} and \ref{theorem-oppositeii} is not essential when considering the dual equation is not necessary.
\end{remark}

We finish this part of the present paper by observing that in each of the mentioned cases the order of equation \eqref{main-equation} can be esentially lowered. The key property allowing us to use one of the results above is the monotonicity that is implied by the continuity of solutions to considered equation.

\section{A generalized Zolt\'an Boros' problem}\label{boros-section}
In \cite{draga-morawiec} the authors solved the original problem posed by Zolt\'an Boros (see \cite{boros}) during the 50th International Symposium on Functional Equations (2012). Namely, we determined all continuous solutions to equation \eqref{zoltan-equation} for $n=3$. Let us recall this result.

\begin{theorem}\label{dm}
Assume $J\subset(0,\infty)$ to be an interval and let $f\colon J\to J$ be a continuous solution to the equation $f^3(x)=[f(x)]^3/x^2$.
\begin{enumerate}[\rm(i)]
\item If $J$ is bounded and $0\notin\mathrm{cl}\,J$, then $f(x)=x$ for every $x\in J$.
\item If $J$ is bounded and $0\in\mathrm{cl}\,J$, then there exists $c\in(0,1]$ such that
\begin{equation}\label{draga-morawiecii}
f(x)=cx\quad\mbox{ for every }x\in J.
\end{equation}
\item If $J$ is unbounded and $0\notin\mathrm{cl}\,J$, then there exists $c\in[1,\infty)$ such that \eqref{draga-morawiecii} holds.
\item If $J=(0,\infty)$, then there exists $c\in(0,\infty)$ such that either \eqref{draga-morawiecii} holds or
\[ f(x)=\frac{c}{x^2}\quad\mbox{ for every }x\in(0,\infty). \]
\end{enumerate}
\end{theorem}

The next lemma gives connection between generalized Zolt\'an Boros' problem \eqref{zoltan-equation} and iterative equations of the form \eqref{main-equation}. Its proof is immediate and we omit it.

\begin{lemma}\label{lemma-substitution}
If $J\subset(0,\infty)$ is an interval and $f\colon J\to J$ is a solution to \eqref{zoltan-equation}, then the formula $g=\log\circ f\circ\exp$ defines a function acting from $\log J$ into itself such that the equation
\begin{equation}\label{rownanien}
g^n(x)=ng(x)-(n-1)x
\end{equation}
holds for every $x\in\log J$.

Conversely, if $I\subset \R$ is an interval and $g\colon I\to I$ is a solution to \eqref{rownanien}, then the formula $f=\exp\circ g\circ\log$ defines a function acting from $\exp I$ into itself such that \eqref{zoltan-equation} holds for every $x\in\exp I$.
\end{lemma}

\begin{lemma}\label{lemma-decreasing}
If a continuous and decreasing function $g\colon I\to I$ satisfies equation \eqref{rownanien}, then $I=\R$. In particular, $g$ is surjective.
\end{lemma}

\begin{proof}
For an indirect proof suppose that $I\neq\R$. Put $a=\inf I$ and $b=\sup I$. If it were $a=-\infty$, then
\[ \infty=\lim_{x\to a}(ng(x)-(n-1)x)=\lim_{x\to a}g^n(x)\le b<\infty, \]
a contradiction. Similarly, if it were $b=\infty$, then
\[ -\infty<a\le\lim_{x\to b}g^n(x)=\lim_{x\to b}(ng(x)-(n-1)x)=-\infty, \]
a contraciction. Thus $a,b\in\R$. Put $c=\inf g(I)$. Since $g$ is decreasing, we have $a\le c<b$. Further, we obtain
\[ c\le\lim_{x\to b}g^n(x)=\lim_{x\to b}(ng(x)-(n-1)x)=nc-(n-1)b, \]
which is an obvious contradiction to $c<b$.

For the part \textit{in particular} see Remark \ref{remark-surjective}.
\end{proof}

The characteristic equation of \eqref{rownanien} is of the form
\begin{equation}\label{zoltan-characteristic}
r^n=nr-n+1.
\end{equation}
Since
\[ r^n-nr+n-1=(r-1)^2\left(\sum_{k=1}^{n-1}kr^{n-1-k}\right), \]
it follows that $1$ is a root of equation \eqref{zoltan-characteristic} of multiplicity $2$. We will need two lemmas on the behaviour of the roots of this equation.

\begin{lemma}\label{behavior1}
\begin{enumerate}[\rm (i)]
\item If $n\in2\N$, then equation \eqref{zoltan-characteristic} has no root in  $\R\setminus\{1\}$.
\item If $n\in 2\N+1$, then equation \eqref{zoltan-characteristic} has exactly one root $r_0\in\R\setminus\{1\}$. Moreover, $r_0$ is a single root and if $z\in\C\setminus\R$ is a root of equation \eqref{zoltan-characteristic}, then $|z|<-r_0$.
\end{enumerate}
\end{lemma}

\begin{proof}
Define a function $f\colon\R\to\R$ by $f(x)=x^n-nx+n-1$. Then
$f'(x)=nx^{n-1}-n$ and $f''(x)=n(n-1)x^{n-2}$ for every $x\in\R$.

{\rm (i)} If $n$ is even, then the function $f$ is convex and has the global minimum at the point $x=1$. Hence $f(x)>f(1)=0$ for every $x\neq 1$.

{\rm (ii)} If $n$ is odd, then the function $f$ has a local maximum at the point $x=-1$ and a local minimum at the point $x=1$. Hence there exists a unique point $r_0\neq 1$ such that $f(r_0)=0$. Clearly, $r_0<-1$. Since $f'(r_0)\neq 0$, we conclude that $r_0$ is a single root of equation \eqref{zoltan-characteristic}.

Let $n=2k+1\ge 5$ with $k\in\N$. Consider a root $z\in\C\setminus\R$ of equation \eqref{zoltan-characteristic} and set all complex roots of this equation in the following sequence $(1,1,r_0,z,\xbar{z},w_1,\xbar{w_1},...,w_{k-2},\xbar{w_{k-2}})$; if $n=5$, then the sequence consists only of first five elements. By Vieta's formulas we have $r_0|z w_1\ldots w_{k-2}|^2=-2k$. This jointly with \eqref{zoltan-characteristic} yields
\[ r_0^{2k}=2k+1+|z w_1\ldots w_{k-2}|^2\quad\mbox{ and }\quad z^{2k}=2k+1+r_0\:\xbar{z}\:|w_1\ldots w_{k-2}|^2. \]
Suppose that, contrary to our claim, $-r_0\le |z|$. Then
\begin{align*}
|z|^{2k} &=\big|2k+1+r_0\:\xbar{z}\:|w_1\ldots w_{k-2}|^2\big|<2k+1-r_0|\:\xbar{z}\:|\!\cdot\!|w_1\ldots w_{k-2}|^2\\
&\le2k+1+|z w_1\ldots w_{k-2}|^2=r_0^{2k},
\end{align*}
a contradiction.
\end{proof}

\begin{lemma}\label{behavior2}
If $z\in\C\setminus\R$ is a root of equation \eqref{zoltan-characteristic}, then $|z|>1$.
\end{lemma}

\begin{proof}
Let $z=|z|(\cos\varphi+\rmi\sin\varphi)$. Since $z\notin\R$, we have $|z|>0$ and $\sin\varphi\neq 0$. By \eqref{zoltan-characteristic} we conclude that $|z|^n\sin n\varphi=|z| n\sin\varphi$, and hence that
\[ |z|=\sqrt[n-1]{\frac{n\sin\varphi}{\sin n\varphi}}. \]

Now it is enough to show that 
\begin{equation}\label{zoltan2}
|\sin n\varphi|<n|\sin \varphi|.
\end{equation}
We prove, by induction, that \eqref{zoltan2} holds for all integers $n\ge 2$ and reals $\varphi$ such that $\sin\varphi\neq0$ or, equivalently, $|\cos\varphi|<1$.

Clearly, $|\sin 2\varphi|=2|\sin\varphi\cos\varphi|<2|\sin\varphi|$. Consider $n\ge 2$ and assume that \eqref{zoltan2} holds for every  $\varphi$ with $\sin\varphi\neq0$. Then
\begin{align*}
|\sin (n+1)\varphi| &=|\sin(n\varphi+\varphi)|\leq |\sin n\varphi\cos\varphi+\cos n\varphi\sin\varphi|\\
&\le|\sin n\varphi\cos\varphi|+|\cos n\varphi\sin\varphi|<|\sin n\varphi|+|\sin\varphi|\\
&<n|\sin \varphi|+|\sin \varphi|=(n+1)|\sin\varphi|,
\end{align*}
which completes the proof.
\end{proof}

Now we can formulate and prove the main result of this section.
\begin{theorem}\label{zoltan-theorem}
If a continuous function $g\colon I\to I$ satisfies equation \eqref{rownanien}, then either it satisfies the equation
\begin{equation}\label{zoltan-increasing}
g^2(x)-2g(x)+x=0
\end{equation}
in the case where $g$ is increasing or it satisfies the equation
\begin{equation}\label{zoltan-decreasing}
g^2(x)-(r_0+1)g(x)+r_0 x=0,
\end{equation}
where $r_0$ stands for the negative root of the polynomial $r^n-nr+(n-1)$, in the case where $g$ is decreasing.
\end{theorem}

\begin{proof}
If $n$ is an even number, then by Lemma \ref{behavior1} the only real root of \eqref{zoltan-characteristic} is $1$ (whose multiplicity equals $2$) and by Lemma \ref{behavior2} all non-real roots have modulus greater than $1$. Theorem \ref{complex-thm1} yields \eqref{zoltan-increasing} to hold. In this case equation \eqref{rownanien} forces directly $g$ to be increasing.

If $n$ is an odd number, then by Lemma \ref{behavior1} equation \eqref{zoltan-characteristic} has two real roots: $1$ of multiplicity $2$ and a single negative root $r_0$. Assuming $g$ to be increasing and using Theorem \ref{theorem-opposite1a} (and also Remark \ref{remark-opposite}) we can eliminate $r_0$ from characteristic equation \eqref{zoltan-characteristic}. Furthermore, using Theorem \ref{complex-thm1}, we can eliminate all non-real roots. Consequently, \eqref{zoltan-increasing} holds. If $g$ is decreasing, then by Lemma \ref{lemma-decreasing} it is surjective. Now using Theorems \ref{theorem-opposite1b} and \ref{complex-thm2}, we eliminate single root $1$ and all non-real roots from \eqref{zoltan-characteristic}. Finally, we conclude that \eqref{zoltan-decreasing} holds.
\end{proof}

The next theorem extends Theorem \ref{dm} and solves the problem of Zolt\'an Boros in the general case.
\begin{theorem}
Assume $J\subset(0,\infty)$ to be an interval and let $f\colon J\to J$ be a continuous solution to equation \eqref{zoltan-equation}.
\begin{enumerate}[\rm(i)]
\item If $J$ is bounded and $0\notin\mathrm{cl}\,J$, then $f(x)=x$ for every $x\in J$.
\item If $J$ is bounded and $0\in\mathrm{cl}\,J$, then there exists $c\in(0,1]$ such that
\begin{equation}\label{draga-morawiecii-new}
f(x)=cx\quad\mbox{ for every }x\in J.
\end{equation}
\item If $J$ is unbounded and $0\notin\mathrm{cl}\,J$, then there exists $c\in[1,\infty)$ such that \eqref{draga-morawiecii-new} holds.
\item If $J=(0,\infty)$ and $n$ is even, then there exists $c\in(0,\infty)$ such that \eqref{draga-morawiecii-new} holds.
\item If $J=(0,\infty)$ and $n$ is odd, then there exists $c\in(0,\infty)$ such that either \eqref{draga-morawiecii-new} holds or
\[ f(x)=cx^{r_0}\quad\mbox{ for every }x\in(0,\infty), \]
where $r_0$ is the negative root of \eqref{zoltan-characteristic}.
\end{enumerate}
\end{theorem}

\begin{proof}
By substitution $g=\log\circ f\circ\exp$ and Lemma \ref{lemma-substitution} we can reduce equation \eqref{zoltan-equation} to \eqref{rownanien}, where $g$ acts from $I=\log J$ into $I$.

In assertions (i)--(iii), by Lemma \ref{lemma-decreasing}, $g$ is increasing. By Theorem \ref{zoltan-theorem} $g$ satisfies \eqref{zoltan-increasing}. Since $r^2-2r+1$ divides $r^3-3r+2$, the function $g$ satisfies \eqref{rownanien} with $n=3$. Consequently, $f$ satisfies \eqref{zoltan-equation} with $n=3$. Therefore, by Theorem \ref{dm}, $f(x)=cx$ for some positive $c$. Since $f$ ranges in $J$, it forces $c=1$ in case (i), $c\in(0,1]$ in case (ii) and $c\in[1,\infty)$ in case (iii). One may easily verify that each of these functions solves \eqref{zoltan-equation}.

In assertion (iv), by equation \eqref{rownanien} $g$ is increasing and by Theorem \ref{zoltan-theorem} it satisfies \eqref{zoltan-increasing}. The same reasoning as above shows that $f$ is given by \eqref{draga-morawiecii-new}. One may easily verify that such an $f$ solves \eqref{zoltan-equation} for arbitrary positive $c$.

In assertion (v) the increasing solutions are obtained exactly in the same manner as in the preceding item. Assume $g$ to be decreasing. Then, by Theorem \ref{zoltan-theorem}, it satisfies \eqref{zoltan-decreasing}. According to \cite{nabeya}, we conclude that $g(x)=r_0x+c_0$ for some constant $c_0$ and, consequently, $f(x)=cx^{r_0}$ for some positive $c$. Again such an $f$ solves \eqref{zoltan-equation} for arbitrary $c\in(0,\infty)$.
\end{proof}

\section{Questions and remarks}
By \cite{matkowski-zhang2} and \cite{nabeya} if characteristic equation \eqref{characteristic-equation} has at least two distinct positive roots or two real opposite roots, then solution to \eqref{main-equation} is not unique (it depends on an arbitrary function). Results from Sections \ref{nonreal-elimination} and \ref{opposite-elimination} allow us to determine solutions to \eqref{main-equation}, provided that they are unique, in many cases. Of particular interest is the situation, still unsolved in full generality, when the characteristic equation has non-real roots. In view of our results it is natural to ask the following question.

\begin{problem}
Set $\varrho=\max\{|r|\colon r\mbox{ is a root of }\eqref{characteristic-equation}\}$. Is it possible to eliminate non-real roots with modulus $\varrho$ from characteristic equation \eqref{characteristic-equation} when this equation has also a real root with modulus $\varrho$?
\end{problem}

It seems that non-real roots of the characteristic equation do not affect solutions to polynomial-like iterative equations in any case, but this problem is still open (see \cite{zhang-zhang}).

\subsection*{Acknowledgement}
The research was supported by the University of Silesia Mathematics Department (Iterative Functional Equations and Real Analysis program).

\end{document}